\documentclass[12pt, reqno, a4paper]{amsart}

\usepackage{amssymb}
\usepackage[bookmarks=false]{hyperref}

\theoremstyle{plain}

\newtheorem*{corollary}{Corollary}
\newtheorem{lemma}{Lemma}
\newtheorem{theorem}{Theorem}
\newtheorem*{conjecture}{Conjecture}

\theoremstyle{remark}
\newtheorem*{remark}{Remark}

\theoremstyle{definition}

\newtheorem{example}{Example}

\DeclareMathOperator{\Id}{Id}
\DeclareMathOperator{\id}{id}
\DeclareMathOperator{\ad}{ad}
\DeclareMathOperator{\height}{ht}
\DeclareMathOperator{\Hom}{Hom}
\DeclareMathOperator{\End}{End}

\DeclareMathOperator{\sign}{sign}
\DeclareMathOperator{\Alt}{Alt}
\DeclareMathOperator{\Ann}{Ann}

\begin{document}

\title[Codimensions of representations of Lie algebras]{Codimensions of
polynomial identities of representations of Lie algebras}

\author{A.\,S.~Gordienko}

\address{Memorial University of Newfoundland, St. John's, NL, Canada}
\email{asgordienko@mun.ca}
\keywords{Lie algebra, polynomial identity, codimension, cocharacter, symmetric group,
Young diagram}

\begin{abstract} Consider a representation $\rho\colon L \to \mathfrak{gl}(V)$
where $L$ is a Lie algebra and $V$ is a finite dimensional vector space.
We prove the analog of Amitsur's conjecture
on asymptotic behavior
for codimensions of polynomial identities of $\rho$.
\end{abstract}

\subjclass[2010]{Primary 17B01; Secondary 16R10, 17B10, 20C30.}
\thanks{Supported by post doctoral fellowship
from Atlantic Association for Research
in Mathematical Sciences (AARMS), Atlantic Algebra Centre (AAC),
Memorial University of Newfoundland (MUN), and
Natural Sciences and Engineering Research Council of Canada (NSERC)}

\medskip

\maketitle

\section{Introduction}

In the 80's, conjectures about the asymptotic behaviour
of codimensions of polynomial identities were made
by S.A.~Amitsur and A.~Regev.
Amitsur's conjecture was proved in 1999 by
A.~Giambruno and M.V.~Zaicev~\cite[Theorem~6.5.2]{ZaiGia}
 for associative algebras and in 2002 by M.V.~Zaicev~\cite{ZaiLie}
 for finite dimensional Lie algebras.
  Alongside with polynomial
identities of algebras, polynomial identities of representations are
also important~\cite{Razmyslov, VovsiPlotkin}.
Therefore the question arises whether the conjectures
hold for codimensions of representations.

Denote by $F \left<  X \right> $ the free associative
 algebra on the countable set $X = \{ x_1,
x_2, \ldots \}$,  i.e. the algebra of polynomials
without a constant term in the noncommuting variables $X$
over a field $F$ of characteristic~$0$.
 Let $\rho \colon L \to \mathfrak{gl}(V)$ be a linear representation
 of a Lie algebra $L$ on a
 vector space $V$ over $F$.  Let
$f=f(x_1,\ldots,x_n) \in F\left<  X \right>$.
We say that $f$ is
a \textit{polynomial identity} of $\rho$ if $f(\rho(a_1), \ldots, \rho(a_n))v=0$
for all $a_1, \ldots, a_n \in L$, $v \in V$.
  The set $\Id(\rho)$ of polynomial identities of $\rho$ is
a two-sided ideal of $F \left<  X \right>$.

\begin{example}
Let $\rho \colon \mathfrak{sl}_2(F) \to \mathfrak{gl}(V)$
be any representation of the algebra $\mathfrak{sl}_2(F)$.
Then $$St_4=\sum_{\sigma \in S_4} (\sign \sigma)x_{\sigma(1)} x_{\sigma(2)}
x_{\sigma(3)} x_{\sigma(4)} \in \Id(\rho)$$
where $S_4$ is the symmetric group on $\lbrace 1, 2, 3, 4\rbrace$.
Indeed, in order to verify a multilinear identity
it is sufficient to substitute basis elements.
Since $\dim\mathfrak{sl}_2(F)=3$ and $St_4$
is alternating in $4$ variables, it vanishes under all
evaluations in $\rho(\mathfrak{sl}_2(F))$.
\end{example}

Let $P_n$ be the space of multilinear polynomials
in the noncommuting variables $x_1$, $x_2$, \ldots,
$x_n$. The non-negative integer
 $c_n(\rho) = \dim \frac {P_n}{P_n \cap \Id(\rho)}$
is called the $n$th {\itshape codimension} of the representation~$\rho$.

Another approach to polynomial identities of representations is concerned with representations of Lie algebras in associative algebras.
 Let $\tau \colon L \to A_0$ be a Lie homomorphism where $L$ is a Lie $F$-algebra and $A_0$ is an associative
$F$-algebra. Let
$f=f(x_1,\ldots,x_n) \in F\left<  X \right>$.
We say that $f$ is
a \textit{polynomial identity} of $\tau$ if $f(\tau(a_1), \ldots, \tau(a_n))=0$
for all $a_1$, \ldots, $a_n \in L$.
  Again, the set $\Id(\tau)$ of polynomial identities of $\tau$ is
a two-sided ideal of $F \left<  X \right>$.
The non-negative integer
 $c_n(\tau) = \dim \frac {P_n}{P_n \cap \Id(\tau)}$
is called the $n$th {\itshape codimension} of the representation~$\tau$.

These approaches are equivalent. The first one is a partial case of the second one with $A_0 = \End_F(V)$
and $\tau = \rho$.
The second one is a partial case of the first one since $\Id(\tau) = \Id(\rho)$ where
$\rho \colon L \to \mathfrak{gl}_F (A_0 + F \cdot 1)$,
 $\rho(a) \cdot b=\tau(a)b$, $a \in L$, $b \in A_0 + F \cdot 1$.

 Throughout the paper we use the first approach.

Yu.\,A. Bahturin has suggested to consider the analog of S.\,A.~Amitsur's conjecture
for representations.

\begin{conjecture} There exists
 $\lim\limits_{n\to\infty} \sqrt[n]{c_n(\rho)} \in \mathbb Z_+$.
\end{conjecture}

\begin{theorem}\label{TheoremMain}
Let $\rho \colon L \to \mathfrak{gl}(V)$ be a
representation of a Lie algebra $L$ on a finite dimensional vector space $V$
over a field $F$ of characteristic $0$. Then either
there exist constants $C_1, C_2 > 0$, $r_1, r_2 \in \mathbb R$, $d \in \mathbb N$
such that $C_1 n^{r_1} d^n \leqslant c_n(\rho) \leqslant C_2 n^{r_2} d^n$
for all $n \in \mathbb N$ or the equality $c_n(\rho)=0$
holds for all sufficiently large $n$.
\end{theorem}
\begin{corollary}
The analog of Amitsur's conjecture holds for such representations.
\end{corollary}

Our proof follows the outline of the proof by M.V.~Zaicev~\cite{ZaiLie}.
However, in many cases we need to apply new ideas.

Codimensions of representations do not change upon an extension of the base field.
The proof is analogous to the cases of codimensions of
associative~\cite[Theorem~4.1.9]{ZaiGia} and Lie algebras~\cite[Section~2]{ZaiLie}.
Thus without loss of generality we may assume $F$ to be algebraically closed.

Fix a Levi decomposition $\rho(L)=G+R$ where $G$ is a maximal semisimple
subalgebra of $\rho(L)$ and $R$ is the solvable radical of $\rho(L)$.
Denote by $A$ the associative subalgebra of $\End_F(V)$ generated by $\rho(L)$.
Then $A$ becomes a $G$-module under the natural left action by multiplication.

Consider left associative ideals $I_1$, $I_2$, \ldots, $I_r$,
$J_1$, $J_2$, \ldots, $J_r$, $r \in \mathbb Z_+$, of the algebra $A$ such that $J_k \subseteq I_k$,
satisfying the conditions
\begin{enumerate}
\item $I_k/J_k$ is an irreducible $A$-module or $\dim (I_k/J_k) = 1$;
\item for any $G$-submodules $T_k$
such that $I_k = J_k\oplus T_k$,
 there exist elements $q_1$, \ldots, $q_{r-1} \in A \cup \{1\}$
such that $T_1 q_1 T_2 \ldots T_{r-1} q_{r-1} T_r \ne 0$.
\end{enumerate}

Let $M$ be a left $A$-module. Denote by $\Ann M$ its annihilator in $A$.
Let $$d=d(\rho) = \max \left(\dim \frac{\rho(L)}{\rho(L) \cap \Ann(I_1/J_1) \cap \dots \cap \Ann(I_r/J_r)}\right)$$
where the maximum is found among all $r \in \mathbb Z_+$ and all $I_1, \ldots, I_r$, $J_1, \ldots, J_r$ satisfying Conditions 1--2. We claim that $d(\rho)$ coincides with $d$ from Theorem~\ref{TheoremMain}.

We need several auxiliary lemmas.

\begin{lemma}\label{LemmaLR}
$[\rho(L), R] \subseteq J(A)$ where $J(A)$ is the Jacobson radical of $A$.
\end{lemma}
\begin{proof}
Let $V = V_0 \supseteq V_1 \supseteq V_2 \supseteq \ldots \supseteq V_t = \left\lbrace 0 \right\rbrace$
be a composition chain in $V$.
Then each $V_i/V_{i+1}$ is an irreducible $\rho(L)$-module.
Denote the corresponding homomorphism by $\varphi_i \colon \rho(L) \to
\mathfrak{gl}(V_i/V_{i+1})$. Then by E.~Cartan's theorem~\cite[Proposition~1.4.11]{GotoGrosshans},
$\varphi_i (\rho(L))$ is semisimple or the direct sum of a semisimple ideal and
the center of $\mathfrak{gl}(V_i/V_{i+1})$.
Thus $\varphi_i ([ \rho(L), \rho(L)])$ is semisimple
and $\varphi_i ([ \rho(L), \rho(L)] \cap R) = 0$.
Since $[\rho(L), R] \subseteq [\rho(L), \rho(L)] \cap R$,
we have $\varphi_i([\rho(L), R])=0$ and $[\rho(L), R]V_i \subseteq V_{i+1}$.
Therefore the associative ideal of $A$ generated by $[\rho(L), R]$ is nilpotent
since $a_1 (b_{11} \ldots b_{1k_1}) a_2 (b_{21} \ldots b_{2k_2})
a_3 \ldots a_t$ is a zero operator on $V$
for all  $a_i \in [\rho(L), R]$ and $b_{ij} \in \rho(L)$.
Thus $[\rho(L), R] \subseteq J(A)$.
\end{proof}

\begin{lemma}\label{LemmaRS}
There exists a subspace $S \subseteq R$ such that
$R = S + \rho(L)\cap J(A)$ is
the direct sum of subspaces and $[G,S]=0$.
\end{lemma}
\begin{proof}
There is a natural $\ad$-action of $G$
on $R$: $(\ad g)a = [g,a]$, $g \in G$, $a\in R$.
Since $\rho(L)\cap J(A)$ is a submodule of $R$ and
$G$ is semisimple,
there exists a complementary $G$-submodule $S$.
Thus $[G,S] \subseteq S$. In virtue of Lemma~\ref{LemmaLR},
$[G,S] \subseteq [\rho(L),R] \subseteq \rho(L) \cap J(A)$.
Therefore, $[G,S]=0$.
\end{proof}

\begin{lemma}\label{LemmaIrr}
Let $M$ be an irreducible $A$-module.
Then either $M$ is an irreducible $G$-module or
$\dim M = 1$.
Furthermore, $S$ acts on $M$ by scalar operators.
\end{lemma}
\begin{proof}
Note that $M$ is an irreducible $\rho(L)$-module.
By Lemma~\ref{LemmaLR}, $[\rho(L), R] \subseteq \rho(L) \cap J(A)$.
Therefore $[\rho(L), R]M\subseteq J(A)M = 0$
and, in virtue of Schur's lemma, $S \subseteq R$ acts by scalar operators on
$M$.
Thus all subspaces in $M$ invariant under the action of $G$
are invariant under the action of $\rho(L)$.
\end{proof}

\begin{lemma}\label{LemmaAnnGS}
Let $M$ be an irreducible $A$-module.
Then $$\Ann_{\rho(L)} M = \Ann_{G} M + \Ann_{S} M + (\rho(L) \cap J(A)).$$
\end{lemma}
\begin{proof}
If $M$ is an irreducible $A$-module, then
$J(A)M = 0$.
Thus it is sufficient to prove that if $g+s \in \Ann_{\rho(L)} M$, $g \in G$,
$s \in S$,
then $g \in \Ann_{G} M$ and $s \in \Ann_{S} M$.
 Denote $\varphi \colon {\rho(L)} \to \mathfrak{gl}(M)$.
Note that $sm=-gm$ for all $m \in M$
since $g+s \in \Ann_{\rho(L)} M$.
 By Lemma~\ref{LemmaIrr}, $s$ is a scalar operator on $M$.
Therefore, $g$ is a scalar operator on $M$ too and $\varphi(g)$
belongs to the center of the semisimple algebra $\varphi(G)$.
Hence $\varphi(g)=0$ and this finishes the proof of the lemma.
\end{proof}

One of the main tools in the investigation of polynomial
identities is provided by the representation theory of symmetric groups.
 The symmetric group $S_n$ acts
 on the space $\frac {P_n}{P_n \cap \Id(\rho)}$
  by permuting the variables.
  Irreducible $FS_n$-modules are described by partitions
  $\lambda=(\lambda_1, \ldots, \lambda_s)\vdash n$ and their
  Young diagrams $D_\lambda$.
   The character $\chi_n(\rho)$ of the
  $FS_n$-module $\frac {P_n}{P_n \cap \Id(\rho)}$ is called the $n$th
  \textit{cocharacter} of the representation $\rho$.
  We can rewrite it as
  a sum $\chi_n(\rho)=\sum_{\lambda\vdash n} m(\rho, \lambda) \chi(\lambda)$ of
  irreducible characters $\chi(\lambda)$.
Let  $e_{T_{\lambda}}=a_{T_{\lambda}} b_{T_{\lambda}}$
and
$e^{*}_{T_{\lambda}}=b_{T_{\lambda}} a_{T_{\lambda}}$
where
$a_{T_{\lambda}} = \sum_{\pi \in R_{T_\lambda}} \pi$
and
$b_{T_{\lambda}} = \sum_{\sigma \in C_{T_\lambda}} (\sign \sigma) \sigma$,
be the Young symmetrizers corresponding to a Young tableau~$T_\lambda$.
Then $M(\lambda) = FS e_{T_\lambda} \cong FS e^{*}_{T_\lambda}$
is an irreducible $FS_n$-module corresponding to a partition~$\lambda \vdash n$.
  We refer the reader to~\cite{ZaiGia, Bahturin, DrenKurs} for an account
  of $S_n$-representations and their applications to polynomial
  identities.

\section{Upper bound}

Fix a composition chain of $A$-submodules
$$A=B_0 \supsetneqq B_1 \supsetneqq B_2 \supsetneqq \ldots \supsetneqq
J(A)\supsetneqq \ldots \supsetneqq B_{\theta-1}
 \supsetneqq B_\theta = \{0\}$$
 in the module $A$. Let $\height a = \max_{a \in B_k} k$ for $a \in A$.

\begin{remark}
If $d=d(\rho)=0$, then $\rho(L) \subseteq \Ann(B_{i-1}/B_i)$
for all $1 \leqslant i \leqslant \theta$ and
 $a_1 a_2 \ldots a_n =0$ for all $a_i \in \rho(L)$
 and $n \geqslant \theta +1$. Thus $c_n(\rho)=0$
 for all $n \geqslant \theta +1$. Therefore we assume $d > 0$.
\end{remark}

 Let $Y=\lbrace y_{11}, y_{12}, \ldots, y_{1j_1};\,
 y_{21}, y_{22}, \ldots, y_{2j_2}; \ldots;\,
 y_{m1}, y_{m2}, \ldots, y_{mj_m}\rbrace$,
 $Y_1$, \ldots, $Y_q$, and $\lbrace z_1, \ldots, z_m\rbrace$
 be subsets of $\lbrace x_1, x_2, \ldots, x_n\rbrace$
 such that $Y_i \subseteq Y$, $|Y_i|=d+1$, $ Y_i \cap Y_j = \varnothing$
 for $i \ne j$,
 $Y \cap \lbrace z_1, \ldots, z_m\rbrace = \varnothing$, $j_i \geqslant 0$.
 Denote $$g_{m,q}=\Alt_{1} \ldots \Alt_{q} ((y_{11}y_{12} \ldots y_{1j_1}) z_1
 (y_{21}y_{22}\ldots y_{2j_2}) z_2 \ldots (y_{m1}y_{m2} \ldots y_{mj_m}) z_m)$$
 where $\Alt_i$ is the operator of alternation on the variables of $Y_i$.

Let $\varphi \colon F\langle X \rangle \to A$
be a homomorphism induced by a substitution $\lbrace x_1, x_2, \ldots, x_n \rbrace \to \rho(L)$.
We say that $\varphi$ is \textit{proper} for  $g_{m,q}$ if
 $\varphi(z_i) \in \rho(L) \cap J(A)$ for $1\leqslant i \leqslant m-1$,
  $\varphi(z_m) \in (\rho(L) \cap J(A)) \cup G \cup S$, and
  $\varphi(y_{ik})\in G \cup S$ for $1\leqslant i \leqslant m$, $1 \leqslant k \leqslant j_i$.

\begin{lemma}\label{LemmaReduct}
Let $\varphi$ be a \textit{proper} homomorphism for $g_{m,q}$.
Then $\varphi(g_{m,q})$ can be rewritten as a
sum of $\psi(g_{m+1,q'})$ where $\psi$
is a proper homomorphism for $g_{m+1,q'}$, $q' \geqslant q - (\dim A)m - 2$.
  ($Y'$, $Y'_i$, $z'_1, \ldots, z'_{m+1}$ may be different
for different terms.)
\end{lemma}
\begin{proof}
Let $\alpha_i = \height \varphi(z_i)$.
We will use induction on $\sum_{i=1}^m \alpha_i$.
(The sum will grow.)
 Note that $ \alpha_i \leqslant \theta \leqslant \dim A$.
Denote $I_i = B_{\alpha_i}$, $J_i = B_{\alpha_{i}+1}$.

First, consider the case when $I_1, \ldots, I_m$,
$J_1, \ldots, J_m$ do not satisfy Conditions 1--2.
In this case we can choose $G$-submodules
$T_i$, $I_i = T_i \oplus J_i$, such that
$T_1 h_1 T_2 h_2 \ldots h_{m-1} T_m = 0$
for all $h_i \in A \cup \lbrace 1 \rbrace$.
Rewrite $\varphi(z_i)=a'_i+a''_i$, $a'_i \in T_i$, $a''_i \in J_i$.
Note that $\height a''_i > \height \varphi(z_i)$.
Since $g_{m,q}$ is multilinear, we can rewrite
$\varphi(g_{m,q})$ as a sum of similar
terms $\tilde\varphi(g_{m,q})$ where $\tilde\varphi(z_i)$ equals either
$a'_i$ or $a''_i$. The term where all
$\tilde\varphi(z_i)=a'_i \in T_i$, equals $0$ since
$I_1, \ldots, I_m$,
$J_1, \ldots, J_m$ do not satisfy Conditions 1--2.
For the other terms $\tilde\varphi(g_{m,q})$ we
have $\sum_{i=1}^m \height \tilde\varphi(z_i) > \sum_{i=1}^m \height \varphi(z_i)$.

Thus without lost of generality
we may assume that $I_1, \ldots, I_m$,
$J_1, \ldots, J_m$ satisfy Conditions 1--2.
In this case, $\dim(\rho(L) \cap \Ann(I_1/J_1) \cap \ldots \cap \Ann(I_m/J_m))
\geqslant \dim(\rho(L))-d$.
In virtue of Lemma~\ref{LemmaAnnGS},
$$\rho(L) \cap \Ann(I_1/J_1) \cap \ldots \cap \Ann(I_m/J_m)
= $$ $$G \cap \Ann(I_1/J_1) \cap \ldots \cap \Ann(I_m/J_m)\
+\ S \cap \Ann(I_1/J_1) \cap \ldots \cap \Ann(I_m/J_m)\ +\ \rho(L) \cap J(A).$$
Choose a basis in $G$ that includes a basis of $G \cap \Ann(I_1/J_1) \cap \ldots \cap \Ann(I_m/J_m)$ and a basis in $S$ that includes a basis of $S \cap \Ann(I_1/J_1) \cap \ldots \cap \Ann(I_m/J_m)$. Since $g_{m,q}$ is multilinear, we may assume
that only basis elements are substituted for $y_{k\ell}$. Note that $g_{m,q}$
is alternating in $Y_i$. Hence, if $\varphi(g_{m,q})\ne 0$, then for every $1 \leqslant i \leqslant q$
there exists $y_{k\ell} \in Y_i$ such that
either $\varphi(y_{k\ell}) \in G \cap \Ann(I_1/J_1) \cap \ldots \cap \Ann(I_m/J_m)$
or $\varphi(y_{k\ell}) \in S \cap \Ann(I_1/J_1) \cap \ldots \cap \Ann(I_m/J_m)$.

Consider the case when $\varphi(y_{k\ell}) \in G \cap \Ann(I_1/J_1) \cap \ldots \cap \Ann(I_m/J_m)$
for some $y_{k\ell}$.
We can choose $G$-submodules $T_k$ such that
$I_k = T_k \oplus J_k$. We may assume
that $\varphi(z_k) \in T_k$ since elements of $J_k$ have greater heights.
Therefore $a\varphi(z_k) \in T_k \cap J_k$ for all $a \in G \cap \Ann(I_1/J_1) \cap \ldots \cap \Ann(I_m/J_m)$.
Hence $a\varphi(z_k)=0$. Moreover, $G \cap \Ann(I_1/J_1) \cap \ldots \cap \Ann(I_m/J_m)$
is a Lie ideal of $G$ and $[G,S]=0$. Thus
$\varphi(y_{k1}\ldots y_{kj_k}z_k) = 0$.
Expanding the alternations, we obtain $\varphi(g_{m,q})=0$.

Consider the case when $\varphi(y_{k\ell}) \in S \cap \Ann(I_1/J_1) \cap \ldots \cap \Ann(I_m/J_m)$
for some $y_{k\ell} \in Y_q$. Expand the alternation $\Alt_q$ in $g_{m,q}$
and rewrite $g_{m,q}$ as a sum of
$$\tilde g_{m,q-1} =\Alt_{1} \ldots \Alt_{q-1} ((y_{11}y_{12} \ldots y_{1j_1}) z_1
 (y_{21}y_{22}\ldots y_{2j_2}) z_2 \ldots (y_{m1}y_{m2} \ldots y_{mj_m}) z_m).$$
 The operator $\Alt_q$ may change indices, however we
 keep the notation $y_{k\ell}$ for the variable with the property
 $\varphi(y_{k\ell}) \in S \cap \Ann(I_1/J_1) \cap \ldots \cap \Ann(I_m/J_m)$.
Now the alternation does not affect $y_{k\ell}$.
Note that $$y_{k1}y_{k2}\ldots y_{k j_k} z_k
= y_{k1} y_{k2} \ldots y_{k,{\ell-1}}
y_{k,{\ell+1}}\ldots y_{k j_k} y_{k\ell} z_k
+$$ $$
\sum\limits_{\beta=\ell+1}^{j_k}
y_{k1} y_{k2} \ldots y_{k,{\ell-1}}
 y_{k,{\ell+1}}
\ldots y_{k,{\beta-1}} [y_{k\ell},y_{k\beta}]
y_{k,{\beta+1}}\ldots y_{kj_k} z_k.$$
In the first polynomial on the right hand side we replace $y_{k\ell} z_k$
with $z'_k$ and define $\varphi'(z'_k)
= \varphi(y_{k\ell}) \varphi(z_k)$, $\varphi'(x) = \varphi(x)$ for other
variables~$x$. Then $\height\varphi'(z'_k) >
\height\varphi(z_k)$ and we can use the inductive assumption.
If $y_{k\beta}\in Y_j$ for some $j$, then we expand the alternation $\Alt_j$
in this term in $\tilde g_{m,q-1}$.
If $\varphi(y_{k\beta}) \in G$, then this term is zero.
If $\varphi(y_{k\beta}) \in S$, then $\varphi([y_{k\ell},y_{k\beta}]) \in \rho(L) \cap J(A)$.
We replace $[y_{k\ell},y_{k\beta}]$ with an additional variable $z'_{m+1}$
and define $\psi(z'_{m+1})=\varphi([y_{k\ell},y_{k\beta}])$, $\psi(x)=\varphi(x)$
for other variables $x$.
Then the polynomial has the desired form. In each inductive step
we reduce $q$ no more than by $1$ and the maximal number of
inductive steps equals $(\dim A)m$. This finishes the proof.
\end{proof}

Since the Jacobson radical is a nilpotent ideal,
 $J(A)^{p} = 0$ for some $p\in \mathbb N$.

\begin{lemma}\label{LemmaUpper}
If $\lambda = (\lambda_1, \ldots, \lambda_s) \vdash n$
and $\lambda_{d+1} \geqslant p((\dim A)p+3)$ or $\lambda_{\dim\rho(L)+1} > 0$, then
$m(\rho, \lambda) = 0$.
\end{lemma}

\begin{proof}
It is sufficient to prove that $e^{*}_{T_\lambda} f \in \Id(\rho)$
for every $f\in P_n$ and a Young tableau $T_\lambda$, $\lambda \vdash n$, with
$\lambda_{d+1} \geqslant p((\dim A)p+3)$ or $\lambda_{\dim\rho(L)+1} > 0$.

Fix some basis of $\rho(L)$ that is a union of
bases of $G$, $S$, and $\rho(L) \cap J(A)$.
Since polynomials are multilinear, it is
sufficient to substitute only basis elements.
 Note that
$e^{*}_{T_\lambda} = b_{T_\lambda} a_{T_\lambda}$
and $b_{T_\lambda}$ alternates the variables of each column
of $T_\lambda$. Hence if we make a substitution and $
e^{*}_{T_\lambda} f$ does not vanish, then this implies that different basis elements
are substituted for the variables of each column.
But if $\lambda_{\dim\rho(L)+1} > 0$, then the length of the first column is greater
than $\dim\rho(L)$. Therefore, $e^{*}_{T_\lambda} f \in \Id(\rho)$.

Consider the case $\lambda_{d+1} \geqslant p((\dim A)p+3)$.
 Let $\varphi$ be a substitution of basis elements for the variables
 $x_1, \ldots, x_n$.
Then $e^{*}_{T_\lambda}f$ can be rewritten as a sum of polynomials $g_{m,q}$
where $1 \leqslant m \leqslant p$, $q \geqslant p((\dim A)p+2)$, and
$z_i$, $1\leqslant i \leqslant m-1$, are replaced with
 elements of $\rho(L) \cap J(A)$.  (For different terms $g_{m,q}$,
 numbers $m$ and $q$,
 variables $z_i$, $y_{ij}$, and sets $Y_i$ can be different.)
 Indeed, we expand symmetrization on all variables and alternation on
 the variables replaced with $\rho(L) \cap J(A)$. Also we expand the alternation
 on the set that includes the last variable, which we denote by $z_m$.

 Applying Lemma~\ref{LemmaReduct} many times, we increase $m$.
 The ideal $J(A)$ is nilpotent and $\varphi(g_{p+1,q})=0$
 for every $q$ and proper homomorphism~$\varphi$.
  Reducing $q$ no more than by $p((\dim A)p+2)$,
 we obtain $\varphi(e^{*}_{T_\lambda}f)=0$.
\end{proof}
Now we can prove
\begin{theorem}\label{TheoremUpper} If $d > 0$, then
there exist constants $C_2 > 0$, $r_2 \in \mathbb R$
such that $c_n(\rho) \leqslant C_2 n^{r_2} d^n$
for all $n \in \mathbb N$. In the case $d=0$, the equality $c_n(\rho)=0$
holds for all sufficiently large $n$.
\end{theorem}
\begin{proof}
Lemma~\ref{LemmaUpper} and~\cite[Lemmas~6.2.4, 6.2.5]{ZaiGia}
imply
\begin{equation}\label{SomeBound}
\sum_{m(\rho, \lambda)\ne 0} \dim M(\lambda) \leqslant C_3 n^{r_3} d^n
\end{equation}
for some constants $C_3, r_3 > 0$.
Denote by $\Id(A) \subseteq F\langle X \rangle$ the ideal of polynomial identities of the associative
algebra~$A$. Let $\chi_n(A)=\chi\left(\frac{P_n}{P_n \cap \Id(A)}\right)
= \sum_{\lambda} m(A, \lambda) \chi(\lambda)$.
Then $\Id(A) \subseteq \Id(\rho)$ and $m(\rho, \lambda) \leqslant m(A, \lambda)
\leqslant C_4 n^{r_4}$ for some $C_4, r_4 > 0$ by~\cite[Lemma~4.9.2]{ZaiGia}.
Together with~(\ref{SomeBound}) this implies the upper bound.
\end{proof}

\section{Lower bound}

By the definition of $d$, there exist left associative ideals $I_1$, $I_2$, \ldots, $I_r$,
$J_1$, $J_2$, \ldots, $J_r$, $r \in \mathbb Z_+$, of the algebra $A$,
satisfying Conditions 1--2, $J_k \subseteq I_k$, such that
$$d = \dim \frac{\rho(L)}{\rho(L) \cap \Ann(I_1/J_1) \cap \dots \cap \Ann(I_r/J_r)}.$$
We consider the case $d > 0$.

Without loss of generality we may assume that
$$\rho(L) \cap \bigcap\limits_{k=1}^r \Ann(I_k/J_k) \ne
\rho(L) \cap \bigcap\limits_{\substack{\phantom{,}k=1,\\ k\ne\ell}}^r \Ann(I_k/J_k)$$
for all $1 \leqslant \ell \leqslant r$.
In particular, $\rho(L)$ has nonzero action on each $I_k/J_k$.

Our aim is to present a partition $\lambda \vdash n$
with $m(\rho, \lambda)\ne 0$ such that $\dim M(\lambda)$
has the desired asymptotic behavior.
We will glue alternating polynomials constructed
by Yu.P.~Razmyslov's theorem
from faithful irreducible modules
over reductive algebras. In order to do this,
we have to choose the reductive algebras.

\begin{lemma}\label{LemmaChooseai}
There exist Lie ideals $G_1, \ldots, G_r$
in $G$ and elements $a_1,\ldots,a_r \in S$
(some of $a_i$ and $G_j$ may be zero)
such that
\begin{enumerate}
\item $G_1+ \ldots + G_r=G_1\oplus \ldots \oplus G_r$;
\item $\langle a_1, a_2, \ldots, a_r \rangle_F = \langle a_1 \rangle_F \oplus \ldots \oplus \langle a_r\rangle_F$;
\item $\sum\limits_{k=1}^r(\dim (G_k\oplus \langle a_k \rangle_F) = d$;
\item $I_k/J_k$ is a faithful $G_k\oplus \langle a_k \rangle_F$-module;
\item $G_i I_k/J_k = a_i I_k/J_k = 0$ for $i > k$.
\end{enumerate}
\end{lemma}
\begin{proof}
Consider $N_\ell = \rho(L) \cap \bigcap\limits_{k=1}^\ell \Ann (I_k/J_k)$,
$1 \leqslant \ell \leqslant r$, $N_0 = \rho(L)$.
Since $G$ is semisimple, we can choose such ideals $G_\ell$
that $N_{\ell-1} \cap G = G_\ell \oplus (N_\ell \cap G)$.
Note that $\dim(S/\Ann_S (I_\ell/J_\ell)) \leqslant 1$
since $S$ acts by scalar operators by Lemma~\ref{LemmaIrr}. Thus we
can choose elements $a_\ell$ such
that $N_{\ell-1} \cap S = \langle a_\ell \rangle_F \oplus (N_\ell \cap G)$.
Hence Properties 1, 2, 4, 5 hold.
 By Lemma~\ref{LemmaAnnGS},
 $N_k = N_k \cap G + N_k \cap S + \rho(L) \cap J(A)$.
 Therefore, $$N_{\ell-1} = G_\ell + N_\ell \cap G + \langle a_\ell \rangle_F +
   N_\ell \cap S + \rho(L) \cap J(A)
   = (G_\ell+\langle a_\ell \rangle_F)+N_\ell$$
   (direct sum of subspaces) and Property 3 holds.
\end{proof}

\begin{lemma}\label{LemmaJordanDecomp}
For every $1 \leqslant i \leqslant r$ there exists a decomposition $a_i = b_i + c_i$
such that $c_i \in A$ acts as a diagonalizable operator on $A$ and $b_i \in J(A)$.
Furthermore, elements $c_i$ commute with each other
 and $b_i$ and $c_i$ are polynomials in $a_1, \ldots, a_r$.
\end{lemma}
\begin{proof}
The solvable Lie algebra $R+J(A)$ acts on the algebra $\tilde A=A+F\cdot 1$
by left multiplication. In virtue of the Lie theorem, there exists
 a basis in $\tilde A$ in which all the operators from $R+J(A)$ have
 upper triangular matrices. Denote the corresponding
 embedding $A  \hookrightarrow M_m(F)$ by $\psi$.
 Here $m = \dim \tilde A$.

Let $A_1$ be the associative algebra generated by $a_1, \ldots, a_r$.
Since $\psi(R) \subseteq \mathfrak{t}_m(F)$, we have
$\psi(A_1) \subseteq UT_m(F)$. Here $UT_m(F)$ is the associative algebra
 of upper triangular matrices $m\times m$.
 There is a decomposition $UT_m(F) = Fe_{11}+Fe_{22}+\dots+Fe_{mm}+N$
 where $N = \langle e_{ij} \mid 1 \leqslant i < j \leqslant m \rangle_F$
 is a nilpotent ideal. Thus there is no subalgebras in $A_1$
 isomorphic to $M_2(F)$ and, by the Wedderburn~--- Malcev theorem,
 $A_1 = Fe_1 + \dots + Fe_t + J(A_1)$
 (the direct sum of subspaces) for some idempotents $e_i \in A_1$.
Denote for every $a_i$  its component in $J(A_1)$ by $b_i$
and its component in $Fe_1 \oplus \dots \oplus Fe_t$ by $c_i$.
Note that $e_i$ are commuting diagonalizable operators. Thus they
have a common basis of eigenvectors in $\tilde A$ and $c_i$
are commuting diagonalizable operators too.

We claim that the space $J(A_1)+J(A)$ generates a nilpotent ideal $I$ in $A$.
First, $\psi(J(A_1)), \psi(J(A)) \subseteq UT_m(F)$
and consist of nilpotent elements. Thus, the corresponding
matrices have zero diagonal elements and
$\psi(J(A_1)), \psi(J(A)) \subseteq N$.
Denote $N_k = \langle e_{ij} \mid i+k \leqslant j \rangle_F \subseteq N$.
Then $$N = N_1 \supsetneqq N_2 \supsetneqq \ldots \supsetneqq N_{m-1} \supsetneqq N_m = \lbrace 0\rbrace.$$
Let $\height_N a = k$ if $\psi(a) \in N_k$, $\psi(a) \notin N_{k+1}$.

Recall that $(J(A))^p =0$.
We claim that $I^{m+p} = 0$.
 The space $I^{m+p}$
 is a span of $h_1 j_1 h_2 j_2 \ldots j_{m+p} h_{m+p+1}$
 where $j_k \in J(A_1) \cup J(A)$ and $h_k \in A \cup \lbrace 1\rbrace$.
 If at least $p$ elements $j_k$ belong to $J(A)$,
 then the product equals $0$. Thus we may assume that at least
 $m$ elements~$j_k$ belong to $J_1(A)$.

 Let $j_i \in J(A_1)$, $h_i \in A \cup \lbrace 1\rbrace$.
We prove by induction on $\ell$ that
$j_1 h_1 j_2 h_2 \ldots h_{\ell-1} j_{\ell}$
can be expressed as a sum of $\tilde j_1 \tilde j_2 \ldots \tilde j_\alpha j'_1 j'_2\ldots j'_\beta
a$
where $\tilde j_i \in J(A_1)$, $j'_i \in J(A)$,
$a \in A \cup \lbrace 1\rbrace$,
 and $\alpha+\sum_{i=1}^\beta \height_N j'_i \geqslant \ell$.
 Indeed, suppose that $j_1 h_1 j_2 h_2 \ldots h_{\ell-2} j_{\ell-1}$
 can be expressed as a sum of $\tilde j_1 \tilde j_2 \ldots \tilde j_\gamma j'_1 j'_2\ldots j'_\varkappa
a$
where $\tilde j_i \in J(A_1)$, $j'_i \in J(A)$,
$a \in \tilde A$,
 and $\gamma+\sum_{i=1}^\varkappa \height_N j'_i \geqslant \ell-1$.
 Then
 $j_1 h_1 j_2 h_2 \ldots j_{\ell-1} h_{\ell-1}j_{\ell}$
 is a sum of
 $$\tilde j_1 \tilde j_2 \ldots \tilde j_\gamma j'_1 j'_2\ldots j'_\varkappa
a h_{\ell-1}j_{\ell} =
\tilde j_1 \tilde j_2 \ldots \tilde j_\gamma j'_1 j'_2\ldots j'_\varkappa
[ah_{\ell-1}, j_{\ell}] + \tilde j_1 \tilde j_2 \ldots \tilde j_\gamma j'_1 j'_2\ldots j'_\varkappa
j_{\ell} (a h_{\ell-1}).$$
Note that, in virtue of the Jacobi identity
and Lemma~\ref{LemmaLR},  $[ah_{\ell-1}, j_{\ell}] \in
J(A)$. Thus it is sufficient to consider only the second term
in the right hand side.
However
$$\tilde j_1 \tilde j_2 \ldots \tilde j_\gamma j'_1 j'_2\ldots j'_\varkappa
j_{\ell} (a h_{\ell-1})
= \tilde j_1 \tilde j_2 \ldots \tilde j_\gamma j_{\ell} j'_1 j'_2\ldots j'_\varkappa
 (a h_{\ell-1})+$$ $$\sum_{i=1}^{\varkappa}
\tilde j_1 \tilde j_2 \ldots \tilde j_\gamma  j'_1 j'_2\ldots j'_{i-1}[j'_{i}, j_\ell]
j'_{i+1}\ldots j'_\varkappa (a h_{\ell-1}).$$
Since $[j'_{i}, j_\ell] \in J(A)$ and $\height_N [j'_{i}, j_\ell] \geqslant 1+ \height_N j'_i$,
all the terms have  the desired form.
 Therefore, $$j_1 h_1 j_2 h_2 \ldots j_{m-1} h_{m-1}j_{m}
 \in \psi^{-1}(N_m) = \lbrace 0 \rbrace,$$ $I^{m+p}=0$, and $$
 J(A) \subseteq J(A_1)+J(A) \subseteq I \subseteq J(A).$$ In particular,
$b_i \in J(A_1) \subseteq J(A)$.
\end{proof}

\begin{lemma}\label{LemmaciProperties}
There exist complementary subspaces $I_k=
\tilde T_k \oplus J_k$ such that
\begin{enumerate}
\item $\tilde T_k$ is an irreducible $G\oplus \langle c_1, \ldots, c_r \rangle_F$-submodule;
\item $c_i$ acts on each $\tilde T_k$ as a scalar operator;
\item $\tilde T_k$ is a faithful $G_k\oplus \langle c_k \rangle_F$-module. Moreover,
$\sum\limits_{k=1}^r\dim (G_k\oplus \langle c_k \rangle_F) = d$;
\item $G_i \tilde T_k = c_i \tilde T_k = 0$ for $i > k$.
\end{enumerate}
\end{lemma}
\begin{proof}
The elements $c_i$ are diagonalizable on $A$ and commute. Therefore, an eigenspace
of any $c_i$ is invariant under the action of other $c_k$.
Using induction, we split $A = \bigoplus_{i=1}^\alpha W_i$
where $W_i$ are intersections of eigenspaces of $c_k$
and elements $c_k$ act as scalar operators on $W_i$.
In virtue of Lemma~\ref{LemmaLR} and the Jacobi identity,
 $[c_i, G]=0$. Thus $W_i$ are $G$-submodules
 and $A$ is a completely reducible
 $G\oplus \langle c_1, \ldots, c_r \rangle_F$-module.
  Therefore, for every $J_k$ we can choose
  a complementary  $G\oplus \langle c_1, \ldots, c_r \rangle_F$-submodules
  $\tilde T_k$ in $I_k$.
Then $\tilde T_k \cong I_k/J_k$  is an irreducible $G$-module
by Lemma~\ref{LemmaIrr}. Property 1 is proven.

By Lemma~\ref{LemmaIrr}, for every $a_i$ there exists
$\gamma \in F$ such that for every $h \in I_k$
we have $a_i h = \gamma h + j$ where
$j \in J_k$.
Thus $c_i h  = a_i h - b_i h = \gamma h + (j-b_i h)$
where $j - b_i h \in J_k$. However, $c_i h \in \tilde T_k$
for $ h \in \tilde T_k$,
and $j-b_i h = 0$. Property 2 is proven.

By Lemma~\ref{LemmaChooseai}, each $I_i/J_i$ is a faithful $G_i\oplus\langle a_i\rangle_F$-module
and either $a_i=0$ or $a_i$ acts on $I_i/J_i$ by a nonzero scalar operator.
Thus  either $a_i=c_i=0$ or $a_i \ne 0$ and $c_i \ne 0$ acts on $\tilde T_i$ by a nonzero scalar operator
that belongs to the center of $\End_F \tilde  T_i$.
Since all homomorphic images of $G_i$ have zero center,
each $\tilde T_i$ is a faithful $G_i\oplus\langle c_i\rangle_F$-module
and $\dim (G_k\oplus \langle c_k \rangle_F) = \dim (G_k\oplus \langle a_k \rangle_F)$.
Property 3 is proven.

Property 4 is a consequence of Property
5 of Lemma~\ref{LemmaChooseai}.
\end{proof}


\begin{lemma}\label{ChooseSubmodule}
Let $M$ be a finite dimensional irreducible module
over a semisimple Lie algebra $G$
and $\tilde G$ be a Lie ideal of $G$
such that $M$ is a faithful $\tilde G$-module. Then there exist
faithful irreducible $\tilde G$-submodules
$\tilde M_i \subseteq M$
such that $M = \tilde M_1 \oplus \tilde M_2 \oplus \dots
\oplus \tilde M_k$.
\end{lemma}
\begin{proof}
There exists a semisimple ideal $\bar G$ such that $G = \tilde G \oplus \bar G$.
By~\cite[Proposition 7.3.1']{GotoGrosshans}, $M \cong \tilde M \otimes \bar M$
for some irreducible $\tilde G$-module $\tilde M$
and irreducible $\bar G$-module $\bar M$. Let $m_1, \ldots, m_k$ be basis elements
of~$\bar M$.
Then $M = (\tilde M \otimes m_1)\oplus(\tilde M \otimes m_2)\oplus\dots\oplus
(\tilde M \otimes m_k)$ is the direct sum of irreducible $\tilde G$-modules.
These  submodules are faithful
since $M$ is faithful.
\end{proof}

 By Condition 2 of the definition of $d(\rho)$,
 there exist elements $q_1$, \ldots, $q_{r-1} \in A \cup \{1\}$
such that $\tilde T_1 q_1 \tilde  T_2 \ldots \tilde  T_{r-1} q_{r-1}
\tilde  T_r \ne 0$.
 Choose $n_i \in \mathbb Z_+$ with the maximal $\sum\limits_{i=0}^{r-1} n_i$ such that
$$\left(\prod_{k=1}^{n_0} j_{0k}\right)\tilde T_1 q_1
 \left(\prod_{k=1}^{n_1} j_{1k}\right) \tilde T_2
  \ldots q_{r-1} \left(\prod_{k=1}^{n_{r-1}} j_{r-1,k}\right) \tilde T_r \ne 0
$$ for some $j_{ik}\in J(A)$.
Denote $\left( q_i \prod_{k=1}^{n_i} j_{ik} \right)$ again by
$q_i$ for $1 \leqslant i \leqslant r-1$
and $\prod_{k=1}^{n_0} j_{0k}$ by $q_0$.
Then $q_i  \in A \cup \{1\}$ and
$$q_0 \tilde T_1 q_1 \tilde T_2 \ldots \tilde T_{r-1} q_{r-1}
\tilde T_r \ne 0,$$ but
\begin{equation}\label{EquationJZero}
q_0 \tilde T_1 q_1 \tilde  T_2 \ldots \tilde  T_{k-1} q_{k-1} (j \tilde  T_{k}) q_{k} \tilde T_{k+1} \ldots \tilde  T_{r-1} q_{r-1}
\tilde T_r = 0\end{equation}
for all $j \in J(A)$ and $1 \leqslant k \leqslant r$.

By Lemma~\ref{ChooseSubmodule}, for every $k$ we can choose a
   faithful irreducible
$G_k \oplus \langle c_k \rangle_F$-submodule $T_k \subseteq \tilde T_k$
such that
\begin{equation}\label{EquationqNonZero}q_0 T_1 q_1 T_2 \ldots T_{r-1} q_{r-1}
T_r \ne 0.\end{equation}

\begin{lemma}\label{LemmaChange}
Let $\varphi \colon \bigoplus_{i=1}^r(G_i \oplus \langle a_i \rangle) \to
\bigoplus_{i=1}^r(G_i \oplus \langle c_i \rangle) $
be the isomorphism defined by formulas $\varphi(g)=g$ for
all $g \in G_i$ and $\varphi(a_i)=c_i$.
Let $f_i$ be multilinear polynomials and $h^{(i)}_1, \ldots, h^{(i)}_{n_i}
\in \bigoplus_{k=1}^r(G_k \oplus \langle a_k \rangle_F)$, $t_i \in
\tilde T_i$ be some elements.
Then $$q_0\,f_1(h^{(1)}_1, \ldots, h^{(1)}_{n_1})\, t_1\
q_1\, f_2(h^{(2)}_1, \ldots, h^{(2)}_{n_2})\, t_2
\ldots q_{r-1}\, f_r(h^{(r)}_1, \ldots, h^{(r)}_{n_r})\, t_r
=$$
$$q_0\, f_1(\varphi(h^{(1)}_1), \ldots, \varphi(h^{(1)}_{n_1}))\, t_1\
q_1\, f_2(\varphi(h^{(2)}_1), \ldots, \varphi(h^{(2)}_{n_2}))\, t_2$$ $$
\ldots q_{r-1}\, f_r(\varphi(h^{(r)}_1), \ldots, \varphi(h^{(r)}_{n_r}))\, t_r.
$$ In other words, we can replace $a_i$ with $c_i$ and the result does not change.
\end{lemma}
\begin{proof}
We rewrite $a_i=b_i+c_i=b_i+\varphi(a_i)$ and use the multilinearity
of $f_i$. By~(\ref{EquationJZero}), terms with $b_i$ vanish.
\end{proof}

Let $mk \leqslant n$ where $m,k,n \in \mathbb N$ are some numbers.
 Denote by $Q_{m,k,n} \subseteq P_n$
the subspace spanned by all polynomials that are alternating in
$k$ disjoint subsets of variables $\{x^i_1, \ldots, x^i_m \}
\subseteq \lbrace x_1, x_2, \ldots, x_n\rbrace$, $1 \leqslant i \leqslant k$.

We need the following theorem~\cite[Remark~12.1]{Razmyslov}.

\begin{theorem}[Yu.\,P.~Razmyslov]\label{Razmyslov}
Let $B$ be a reductive Lie algebra over an algebraically closed field
of characteristic zero, and let $\dim B = m$. Let $\tau \colon B \hookrightarrow U$ be an
embedding of $B$ to a simple associative algebra $U$ generated by $\tau(B)$, with a nonzero
center. Then there exist $k$ and a multilinear
polynomial $f \in Q_{m,k,mk} \backslash \Id(\tau)$
such that $f(\tau(p_1), \ldots, \tau(p_{mk}))$ belongs to the center of $U$ for all $p_i \in B$.
\end{theorem}

\begin{lemma}\label{LemmaAlt} If $d \ne 0$, then there exist numbers $k, n_0 \in \mathbb N$ such that for every $n\geqslant n_0$
there exist disjoint subsets $X_1$, \ldots, $X_{k\ell} \subseteq \lbrace x_1, \ldots, x_n
\rbrace$, $\ell = \left[\frac{n-n_0}{kd}\right]$,
$|X_1| = \ldots = |X_{k\ell}|=d$ and a polynomial $f \in P_n \backslash
\Id(\rho)$ alternating on the variables of each set $X_j$.
\end{lemma}

\begin{proof}
Recall that each $T_i$ is a faithful irreducible $G_i\oplus\langle c_i\rangle_F$-module.
Denote by $\tau_i$ the embedding $G_i\oplus\langle c_i\rangle \hookrightarrow \End_F(T_i)$.
In virtue of the density theorem~\cite[Section 4.3]{Jacobson},  the space $\tau_i(G_i\oplus\langle c_i\rangle)$
generates $\End_F(T_i)$ as an associative algebra.
The center of $\End_F(T_i)$ is a linear span of its unit $\id_{T_i}$.
In virtue of Theorem~\ref{Razmyslov}, for some $k_i \in \mathbb N$
there exists a multilinear polynomial $f_i \in Q_{d_i,k_i, k_i d_i} \backslash \Id(\rho_i)$,
$d_i = \dim (G_i\oplus\langle c_i\rangle)$, satisfying the following property.
If we substitute any elements from $\tau_i(G_i\oplus\langle c_i\rangle)$ for the variables of $f_i$,
then the result belongs to the center of $\End_F(T_i)$.
Therefore, $\tau_i(f_i(h^{(i)}_1, \ldots, h^{(i)}_{d_i k_i})) = \id_{T_i}$ for some
$h^{(i)}_\beta \in G_i\oplus\langle c_i\rangle_F$
and $f_i(h^{(i)}_1, \ldots, h^{(i)}_{d_i k_i})t_i = t_i$ for all $t_i\in T_i$. We can
consider products of several copies of $f_i$ in different variables
and assume that $k_1 = k_2 = \ldots = k_r = k$.

Equation~(\ref{EquationqNonZero}) implies
that
\begin{equation}\label{EquationqNonZero2}
 b:=q_0 t_1 q_1 t_2 \ldots q_{r-1} t_r \ne 0
 \end{equation}
for some $t_i \in T_i$. Fix these $t_i$.
Note that $q_i$ are sums of $q_{i1} \ldots q_{i\theta_i}$
and  $t_i$ are sums of $t_{i1} \ldots t_{i\xi_i}$ where $\theta_i \geqslant 0$,
$\xi_i > 0$ and  $q_{ik}, t_{ik}\in\rho(L)$. Denote the maximal
 $\theta_i$ by $\eta_i$ and the maximal
 $\xi_i$ by $\zeta_i$. Let
 $n_0=\sum_{i=1}^r (\eta_{i-1}+\zeta_i)$,
 $\ell=\left[ \frac{n-n_0}{kd}\right]$, and
 $\alpha=\left[ \frac{n-\ell kd}{kd_1}\right]+1$.

 Consider the polynomial
$\tilde f^{(\alpha)} = w_0 f_1^{\alpha}\ \tilde f$ where
$$\tilde f = (f_1^{\ell} z_1)\, w_1
(f_2^{\ell} z_2)\, w_2
(f_3^{\ell} z_3) \ldots (f_{r-1}^{\ell} z_{r-1}) w_{r-1}
(f_r^{\ell} z_r).$$
Here the variables of all copies of $f_i$ are
different and $f^{(\alpha)}$ is a multilinear polynomial of degree
$\ell k d + 2r + \alpha k d_1$.

Let $w_i=q_i$, $z_i=t_i$. Equation (\ref{EquationqNonZero2}) and the choice of $f_i$
imply that if we replace the variables of $f_i$ with $h^{(i)}_\beta$, then
the value of $\tilde f^{(\alpha)}$ equals $b$. Denote this substitution of elements by $\Delta$.

Each copy of $f_i$, $1\leqslant i \leqslant r$,
is alternating on disjoint sets $X^{i,j}_1$, \ldots, $X^{i,j}_k$
where $|X^{i,j}_1|=\ldots = |X^{i,j}_k| = d_i$
and $1 \leqslant j \leqslant \ell$ is the number of the copy of $f_i$ in $\tilde f$.
Denote $X_{(j-1)k+v} = \coprod_{i=1}^r X^{i,j}_v$ for
$1 \leqslant j \leqslant \ell$, $1 \leqslant v \leqslant k$.
Then $|X_j|=d$.
Consider the polynomial $ f^{(\alpha)} =
w_0 f_1^{\alpha}  \Alt_1 \ldots \Alt_{k\ell} \tilde f$
where the operator $\Alt_j$ alternates a polynomial on the variables of $X_j$.
We claim that $f^{(\alpha)}$ does not vanish
under the substitution $\Delta$.

Note that the alternation
does not change $z_i$. Thus Property 4 of Lemma~\ref{LemmaciProperties}
implies that all the terms of alternation
in which variables from $X^{i,j}_v$ are put on the places of
variables from $X^{i',j}_{v}$, $i > i'$, vanish.
Hence we may consider only those
terms where $X^{i,j}_v$ are invariant.
Recall that each $f_i$ is alternating on $X^{i,j}_v$.
Thus the result of the substitution
$\Delta$ for the variables of $f^{(\alpha)}$ equals
$(d_1!d_2! \ldots d_r!)^{k\ell}\, b \ne 0$.
Moreover, by Lemma~\ref{LemmaChange},
it is not important whether we replace the variables of $f_i$ with $a_i$ or $c_i$.
Now we assume that all the variables of $f_i$ are replaced with elements from
$G_i\oplus\langle a_i\rangle$.

Now we rewrite $q_i$ as sums of $q_{i1} \ldots q_{i\theta_i}$
and  $t_i$ as sums of $t_{i1} \ldots t_{i\xi_i}$ where $0 \leqslant \theta_i \leqslant \eta_i$, $0 < \xi_i \leqslant \zeta_i$ and  $q_{ik}, t_{ik}\in\rho(L)$.
Using the linearity of $f^{(\alpha)}$ in $w_i$ and $t_i$, we obtain
that $f^{(\alpha)}$ does not vanish under some substitution
$w_i=q_{i1} \ldots q_{i\theta_i}$ and $z_i=t_{i1} \ldots t_{i\xi_i}$.
Now we replace $w_i$ in $\tilde f$ with $y_{i1} \ldots y_{i\theta_i}$
(or remove $w_i$ if $\theta_i = 0$)
and $z_i$ with $v_{i1} \ldots v_{i\xi_i}$.
Denote the polynomial obtained by $\hat f$.
Then $$\hat f^{(\alpha)} = y_{01}\ldots y_{0\theta_0}
f_1^\alpha \Alt_1 \ldots \Alt_{k\ell} \hat f$$ is a multilinear polynomial of degree
$$\ell k d + \alpha k d_1 +\sum_{i=1}^r (\theta_{i-1} + \xi_i)$$
and $\hat f^{(\alpha)} \notin \Id(\rho)$.
Note that $\ell k d \leqslant \deg \hat f \leqslant n_0 + \ell k d \leqslant n$
and $\deg \hat f^{(\alpha)}  > \ell k d + \alpha k d_1 > n$.

Expand the first $\alpha$ copies of $f_1$ in $\hat f^{(\alpha)}$.
At least one of the polynomials obtained is not a polynomial identity. Denote it by
$g^{(\alpha)}$.
Remove $\left(\deg \hat f^{(\alpha)}-n\right)$ initial variables in $g^{(\alpha)}$
and denote the polynomial obtained by $f$.
Then
$$f = u_1 u_2 \ldots u_{\left(n-\deg \hat f\right)} \Alt_1 \ldots \Alt_{k\ell} \hat f
\in P_n \backslash \Id(\rho)$$
satisfies all the conditions of Lemma~\ref{LemmaAlt}.
Here $u_i$ are some variables from $\lbrace x_1, x_2, \ldots, x_n \rbrace$.
\end{proof}

\begin{lemma}\label{LemmaCochar} Let
 $k, n_0,\ell$ be the numbers from
Lemma~\ref{LemmaAlt}.   Then for every $n \geqslant n_0$ there exists
a partition $\lambda = (\lambda_1, \ldots, \lambda_s) \vdash n$,
$\lambda_i \geqslant k\ell-C$ for every $1 \leqslant i \leqslant d$,
with $m(\rho, \lambda) \ne 0$.
Here $C = p((\dim A)p + 3)((\dim\rho(L))-d)$.
\end{lemma}
\begin{proof}
Consider the polynomial $f$ from Lemma~\ref{LemmaAlt}.
It is sufficient to prove that $e^*_{T_\lambda} f \notin \Id(\rho)$
for some tableau $T_\lambda$ of a desired shape $\lambda$.
It is known that $FS_n = \bigoplus_{\lambda,T_\lambda} FS_n e^{*}_{T_\lambda}$ where the summation
runs over the set of all standard tableax $T_\lambda$,
$\lambda \vdash n$. Thus $FS_n f = \sum_{\lambda,T_\lambda} FS_n e^{*}_{T_\lambda}f
\not\subseteq \Id(\rho)$ and $e^{*}_{T_\lambda} f \notin \Id(\rho)$ for some $\lambda \vdash n$.
We claim that $\lambda$ is of a desired shape.
It is sufficient to prove that
$\lambda_d \geqslant k\ell-C$, since
$\lambda_i \geqslant \lambda_d$ for every $1 \leqslant i \leqslant d$.
Each row of $T_\lambda$ includes numbers
of no more than one variable from each $X_i$,
since $e^{*}_{T_\lambda} = b_{T_\lambda} a_{T_\lambda}$
and $a_{T_\lambda}$ is symmetrizing the variables of each row.
Thus $\sum_{i=1}^{d-1} \lambda_i \leqslant \ell k(d-1) + (n-\ell kd) = n-\ell k$.
In virtue of Lemma~\ref{LemmaUpper},
$\sum_{i=1}^d \lambda_i \geqslant n-C$. Therefore
$\lambda_d \geqslant \ell k-C$.
\end{proof}

\begin{proof}[Proof of Theorem~\ref{TheoremMain}]
The Young diagram~$D_\lambda$ from Lemma~\ref{LemmaCochar} contains
the rectangular subdiagram~$D_\mu$, $\mu=(\underbrace{k\ell-C, \ldots, k\ell-C}_d)$.
The branching rule for $S_n$ implies that if we consider a restriction of
$S_n$-action on $M(\lambda)$ to $S_{n-1}$, then
$M(\lambda)$ becomes the direct sum of all non-isomorphic
$FS_{n-1}$-modules $M(\nu)$, $\nu \vdash (n-1)$, where each $D_\nu$ is obtained
from $D_\lambda$ by deleting one box. In particular,
$\dim M(\nu) \leqslant \dim M(\lambda)$.
Applying the rule $(n-d(k\ell-C))$ times, we obtain $\dim M(\mu) \leqslant \dim M(\lambda)$.
By the hook formula, $\dim M(\mu) = \frac{(d(k\ell-C))!}{\prod_{i,j} h_{ij}}$
where $h_{ij}$ is the length of the hook with edge in $(i, j)$.
By Stirling formula,
$$\dim M(\mu) \geqslant \frac{(d(k\ell-C))!}{((k\ell-C+d)!)^d}
\sim \frac{
\sqrt{2\pi d(k\ell-C)} \left(\frac{d(k\ell-C)}{e}\right)^{d(k\ell-C)}
}
{
\left(\sqrt{2\pi (k\ell-C+d)}
\left(\frac{k\ell-C+d}{e}\right)^{k\ell-C+d}\right)^d
} \sim C_5 \ell^{r_5} d^{dk\ell}$$ as $\ell \to \infty$
for some constants $C_5 > 0$, $r_5 \in \mathbb Q$.
Since $\ell = \left[\frac{n-n_0}{kd}\right]$,
this gives the lower bound.
The upper bound has been proved in Theorem~\ref{TheoremUpper}.
\end{proof}

\section{Examples}

\begin{example}
Let $\rho \colon L \to \mathfrak{gl}(V)$ be a finite dimensional
representation of a Lie algebra $L$
over a field $F$.
If $\rho(a)$ is a nilpotent operator for any $a \in L$,
then $c_n(\rho)=0$ for $n \geqslant \dim V$.
Conversely, if $c_n(\rho)=0$ for some $n\in\mathbb N$, then $\rho(L)$
generates a nilpotent associative subalgebra $A$ in
$\End_F(V)$.
\end{example}
\begin{proof}
Suppose $L$ acts on $V$ by nilpotent operators.
One of the variants of the Engel theorem~\cite[Corollary from Theorem~3.3]{Humphreys}
implies that there exists a basis in $V$ in which
all the operators $\rho(a)$, $a \in L$,
have strictly upper triangular matrices. The associative algebra
of strictly upper triangular matrices is nilpotent
and $\rho(a_1)\rho(a_2)\ldots \rho(a_n)=0$ in $\End_F(V)$
for all $a_i \in L$ and $n \geqslant \dim V$. Thus $c_n(\rho)=0$.

Suppose $c_n(\rho)=0$ for some $n\in\mathbb N$. Then
$\rho(a_1)\rho(a_2)\ldots \rho(a_n)=0$ for all $a_i \in L$.
 Hence $A^n=0$.
\end{proof}

\begin{example}\label{ExampleIrr}
Let $\rho \colon L \to \mathfrak{gl}(V)$ be a finite dimensional
irreducible representation of a Lie algebra $L$
over an algebraically closed field $F$ of characteristic $0$. Then
there exist constants $C > 0$, $r \in \mathbb R$
such that $$C n^r (\dim \rho(L))^n \leqslant c_n(\rho) \leqslant (\dim V)^2(\dim \rho(L))^n \text{ for all } n \in \mathbb N.$$
\end{example}
\begin{proof}
Consider polynomials as $n$-linear maps from $\rho(L)$ to $\End_F V$.
Then we have the natural map $P_n \to \Hom_{F}((\rho(L))^{{}\otimes n}; \End_F V)$
with the kernel $P_n \cap \Id(\rho)$
that leads to the embedding $$\frac{P_n}{P_n \cap \Id(\rho)}
\hookrightarrow \Hom_{F}((\rho(L))^{{}\otimes n}; \End_F V).$$
Thus $$c_n(\rho)=\dim \left(\frac{P_n}{P_n \cap \Id(\rho)}\right)
\leqslant \dim\Hom_{F}((\rho(L))^{{}\otimes n}; \End_F V)=
(\dim V)^2(\dim \rho(L))^n$$
and we obtain the upper bound. Hence $d(\rho) \leqslant \dim \rho(L)$.
Since $V$ is an irreducible $L$-module, $V$ is an irreducible $A$-module.
The density theorem~\cite[Section 4.3]{Jacobson} implies $$A=\End_F(V) \cong M_k(F),\qquad k=\dim V.$$
Let $I=\langle e_{11}, e_{21}, \ldots, e_{k1} \rangle \subseteq A$
and $J=0$. Then $I$ is a minimal left ideal of $A$ and $I/J$ is a
faithful irreducible $A$-module. Thus $\Ann(I/J)=0$
and $$\dim(\rho(L)/(\rho(L)\cap\Ann(I/J)))=\dim(\rho(L)).$$
Hence $d(\rho) \geqslant \dim(\rho(L))$. This yields the lower bound.
\end{proof}

\begin{example}\label{ExampleComplReduc}
Let $\rho \colon L \to \mathfrak{gl}(V)$ be a finite dimensional
completely reducible representation of a Lie algebra $L$
over an algebraically closed field $F$ of characteristic $0$,
i.e.
$$V = \bigoplus_{i=1}^s \bigoplus_{j=1}^{k_i} V_{ij}$$
where each $V_{ij}$ is irreducible and
 $V_{ij}\cong V_{kl}$ if and only if $i=k$. Denote
$\rho_{ij} \colon L \to \mathfrak{gl}(V_{ij})$,
$\rho_{ij}(a)=\rho(a)\bigl|_{V_{ij}}$, $1 \leqslant i \leqslant s$,
$1 \leqslant j \leqslant k_i$,
$a \in L$. Denote $$d=\max_{\substack{1 \leqslant i \leqslant s,
\\ 1 \leqslant j \leqslant k_i}} \left( \dim \rho_{ij}(L) \right).$$ Then
there exist constants $C_1, C_2 > 0$, $r_1, r_2 \in \mathbb R$
such that $$C_1 n^{r_1} d^n \leqslant c_n(\rho) \leqslant C_2 n^{r_2} d^n \text{ for all } n \in \mathbb N.$$
\end{example}
\begin{proof}
Let $A$ be the associative subalgebra of $\End_F(V)$
generated by $\rho(L)$. Note that
$$\Delta=\End_A(V)\cong M_{k_1}(F)\oplus \ldots \oplus M_{k_s}(F).$$
Indeed, denote by $\varphi_{ij}$  the isomorphism $V_{ij}\to V_{i1}$
and by $p_{ij}$ the projection $V \to V_{ij}$.
Let $\psi \in \Delta$.
Then by Schur's lemma $\varphi_{k\ell}p_{k\ell}\psi\varphi_{ij}^{-1}
\colon V_{i1} \to V_{k1}$
is a scalar operator for $i=k$  and is a zero operator
for $i \ne k$. Define $\alpha^{(i)}_{\ell j}(\psi)
\in F$ by the formula
$\varphi_{i\ell}p_{i\ell}\psi\varphi_{ij}^{-1}(a)=\alpha^{(i)}_{\ell j}(\psi)a$
for all $a \in V_{i1}$.
Then $\Phi \colon \Delta \to M_{k_1}(F)\oplus \ldots \oplus M_{k_s}(F)$,
$\Phi(\psi)=\bigl((\alpha^{(1)}_{\ell j}(\psi)), (\alpha^{(2)}_{\ell j}(\psi)),
\ldots, (\alpha^{(s)}_{\ell j}(\psi))\bigr)$, is an isomorphism
and $$\Phi^{-1}\bigl((\beta^{(1)}_{\ell j}), (\beta^{(2)}_{\ell j}),
\ldots, (\beta^{(s)}_{\ell j})\bigr)
= \sum_{i=1}^{s}\sum_{j,\ell=1}^{k_i}
 \beta^{(i)}_{\ell j} \varphi^{-1}_{i\ell}\varphi_{ij}p_{ij}.$$
 Here $(\beta^{(i)}_{\ell j}) \in M_{k_i}(F)$ are arbitrary matrices.

 By the density theorem~\cite[Section 4.3]{Jacobson},
 $A=\End_\Delta(V)$.
Thus $$A=\End_\Delta(V)\cong \End_F(V_{11}) \oplus \End_F(V_{21})
\oplus \ldots \oplus \End_F(V_{s1})$$
where $(\psi_1, \ldots, \psi_s)$ acts on $V_{ij}$ by the
operator $\varphi^{-1}_{ij}\psi_i\varphi_{ij}$. Indeed, every operator $\psi \in \End_\Delta(V)$
must commute with $p_{ij} \in \Delta$. Hence $V_{ij}$ are invariant
under $\psi$. Moreover, $\psi$ must commute with $\varphi^{-1}_{ij}\varphi_{i\ell} p_{i\ell}$.
Therefore, $$\psi(\varphi^{-1}_{ij}\varphi_{i\ell}p_{i\ell}) = (\varphi^{-1}_{ij}\varphi_{i\ell} p_{i\ell})\psi,$$
$$\varphi_{ij}(\psi \varphi^{-1}_{ij}\varphi_{i\ell} p_{i\ell})
\varphi_{i\ell}^{-1} = \varphi_{ij}(\varphi^{-1}_{ij}\varphi_{i\ell} p_{i\ell}\psi)\varphi_{i\ell}^{-1}$$
 and
$\varphi_{ij} \psi \varphi^{-1}_{ij} = \varphi_{i\ell} \psi\varphi_{i\ell}^{-1}$
for all $1\leqslant i \leqslant s$, $1 \leqslant j,\ell \leqslant k_i$.
We obtain the isomorphism $$\Xi \colon A
\to \End_F(V_{11}) \oplus \End_F(V_{21})
\oplus \ldots \oplus \End_F(V_{s1}),$$
 $$\Xi(\psi) = (\psi\bigl|_{V_{11}}, \psi\bigl|_{V_{21}},
 \ldots, \psi\bigl|_{V_{s1}}).$$

 If $I$ is a left ideal in $A$,
 then $I=I_1\oplus\ldots\oplus I_s$ where $\Xi(I_i)$
 is a left ideal in $\End_{F}(V_{i1})$.
 Consider arbitrary left ideals $I_1, \ldots, I_r$,
 $J_1, \ldots, J_r$ of $A$ that satisfy Conditions 1--2.
 There exist numbers $1\leqslant i_1, \ldots, i_r \leqslant s$
 such that $I_k = \tilde I_k \oplus \hat I_k$,
 $J_k = \tilde J_k \oplus \hat I_k$ where $\tilde I_k$ and
 $\tilde J_k$ are left ideals of $\Xi^{-1}(\End_F(V_{i_k1}))$
 and $\hat I_k$ is a left ideal of
 $\Xi^{-1}\left(\bigoplus_{\ell \ne i_k}\End_F(V_{\ell 1})\right)$.
 Hence $\Ann(I_k/J_k)=\Xi^{-1}\left(\bigoplus_{\ell \ne i_k}\End_F(V_{\ell 1})\right)$.
 Thus
 $\Ann(I_k/J_k) \cap \rho(L) = \ker\rho_{i_k1}$
 and $\dim(\rho(L)/(\Ann(I_k/J_k) \cap \rho(L))) = \dim \rho_{i_k1}(L)$.
 Choose $A$-submodules $T_k$ such that $\tilde I_k = T_k \oplus \tilde J_k$.
 If $i_k \ne i_\ell$ for some $k$ and $\ell$, then
 $T_1 q_1 T_2 \ldots T_{r-1} q_{r-1} T_r = 0$
 for all  $q_1$, \ldots, $q_{r-1} \in A \cup \{1\}$.
 Thus $i_1 = \ldots = i_r$ and $d(\rho) = \max_{\substack{1 \leqslant i \leqslant s,\\ 1 \leqslant j \leqslant k_i}} \dim \rho_{ij}(L)$.
  By Theorem~\ref{TheoremMain} we obtain the bounds.
\end{proof}

\begin{example}
Let $\rho \colon L \to \mathfrak{gl}(V)$ be a non-trivial finite dimensional
 representation of a simple Lie algebra $L$
over an algebraically closed field $F$ of characteristic $0$.
Then there exist constants $C_1, C_2 > 0$, $r_1, r_2 \in \mathbb R$
such that $$C_1 n^{r_1} (\dim L)^n \leqslant c_n(\rho) \leqslant C_2 n^{r_2} (\dim L)^n \text{ for all } n \in \mathbb N.$$
\end{example}
\begin{proof} By the Weyl theorem, $\rho$ is completely reducible.
Let $\rho_i \colon L \to \mathfrak{gl}(V_i)$ be the corresponding
irreducible representations. Then Example~\ref{ExampleComplReduc}
implies that $$d(\rho)=\max_{1 \leqslant i \leqslant s}\dim \rho_i(L)=\dim L.$$
\end{proof}

\section{Acknowledgements}

I am grateful to Yuri Bahturin and Mikhail Kotchetov for helpful
discussions.

\end{document}